\documentclass[a4paper,abstract=true,DIV=default,toc=index]{scrartcl}
\KOMAoption{draft}{true}
\KOMAoption{overfullrule}{true}
\KOMAoption{twoside}{true}

\usepackage[normalem]{ulem}

\usepackage[hyperpageref]{backref}
\usepackage{amsmath}
\usepackage{amsthm}
\usepackage{amsfonts}
\usepackage{color,enumitem,graphicx}

\usepackage[english]{babel}
\usepackage[dvipsnames]{xcolor}
\usepackage{hyperref}

\newcommand\myshade{85}
\colorlet{mylinkcolor}{violet}
\colorlet{mycitecolor}{YellowOrange}
\colorlet{myurlcolor}{Aquamarine}

\hypersetup{
	linkcolor  = mylinkcolor!\myshade!black,
	citecolor  = mycitecolor!\myshade!black,
	urlcolor   = myurlcolor!\myshade!black,
	colorlinks = true,
}

\usepackage[T1]{fontenc}
\usepackage{amssymb}
\usepackage{lmodern}
\usepackage[en-us]{datetime2}

\usepackage{fixmath}
\numberwithin{equation}{section}

\newtheorem{theorem}{Theorem}[section]
\newtheorem{lemma}[theorem]{Lemma}
\newtheorem{proposition}[theorem]{Proposition}

\theoremstyle{definition}
\newtheorem{definition}[theorem]{Definition}
\theoremstyle{remark}

\newtheorem{example}[theorem]{Example}

\newcommand{\R}{\mathbb{R}}

\usepackage{todonotes}

\newcommand{\cK}{{\mathcal K}}

\newcommand{\cM}{{\mathcal M}}
\newcommand{\cN}{{\mathcal N}}

\newcommand{\cS}{{\mathcal S}}

\newcommand{\N}{\mathbb N}

\DeclareMathOperator*{\essinf}{ess\, inf}

\numberwithin{equation}{section}

\DeclareMathOperator*{\supp}{supp}
\DeclareMathOperator*{\diam}{diam}

\DeclareSymbolFont{rsfs}{U}{rsfs}{m}{n}
\DeclareSymbolFontAlphabet{\mathscr}{rsfs}

\DeclareMathOperator*{\isom}{Isom}
\DeclareMathOperator*{\fix}{Fix_{\cM}}

\begin{document}

\title{Schr\"odinger equation on Cartan-Hadamard manifolds with oscillating nonlinearities\footnote{The first and third author are supported by
    GNAMPA, project ``Equazioni alle derivate parziali: problemi e
    modelli.''}}

\author{Luigi Appolloni\thanks{Dipartimento di Matematica e
    Applicazioni, Universit\`a degli Studi di Milano Bicocca. Email:
    \href{mailto:l.appolloni1@campus.unimib.it}{l.appolloni1@campus.unimib.it}}
  \and Giovanni Molica Bisci\thanks{Dipartimento di Scienze Pure e
    Applicate, Universit\`{a} di Urbino Carlo Bo. Email:
    \href{mailto:giovanni.molicabisci@uniurb.it}{giovanni.molicabisci@uniurb.it}}
  \and Simone Secchi\thanks{Dipartimento di Matematica e Applicazioni,
    Universit\`a degli Studi di Milano Bicocca. Email:
    \href{mailto:simone.secchi@unimib.it}{simone.secchi@unimib.it}}}

\date{\DTMnow}

\maketitle

\begin{abstract}
	We study the equation $-\Delta_g w+w=\lambda \alpha(\sigma) f(w)$ on a $d$-dimensional homogeneous Cartan-Hadamard Manifold $\cM$ with $d \geq 3$. Without using the theory of topological indices, we prove the existence of infinitely many solutions  for a class of nonlinearities $f$ which have an oscillating behavior either at zero or at infinity.
\end{abstract}

\section{Introduction}
Let $(\cM,g)$ be a $d$-dimensional homogeneous Cartan-Hadamard Manifold with $d \geq 3$. The aim of this paper is to study
\begin{equation*} \label{eqP}
\tag{$P_\lambda$}
\begin{cases}  \displaystyle -\Delta_g w+w=\lambda \alpha(\sigma) f(w) &\hbox{in}\ \cM \\
  w \in H^1_g(\cM)
\end{cases}
\end{equation*}
where $-\Delta_g$ denotes the Laplace-Beltrami operator, $\alpha \in L^1(\cM)  \cap L^\infty (\cM) \setminus \lbrace 0 \rbrace$ is nonnegative, $f :\R \to \R$ is a continuous function and $\lambda >0$ a real parameter.

The stationary nonlinear Schr\"odinger equation is undoubtedly one of the most attractive topics in nonlinear analysis. In the last years many researchers studied this equations under various hypothesis on the non nonlinear term and in different setting. Among them, the study of the nonlinear Schr\"odinger equation on Riemannian manifold has received a particular attention recently. Faraci and Farkas in \cite{MR4020749} using variational methods proved a characterization result for existence of solutions for the Schr\"odinger equation with a divergent potential in a non-compact Riemannian manifold with asymptotically non-negative Ricci curvature. In the same setting of this paper, Molica Bisci and Secchi in \cite{MR3886596} proved some existence and non-existence results for a similar problem, while Appolloni et al. in \cite{https://doi.org/10.48550/arxiv.2203.08482} showed the existence of three critical points for the energy functional associated to a perturbed problem. Krist\'{a}ly in \cite{MR2461832} proved a multiplicity result for the equation without a potential and with $\cM=\mathbb{S}^d$. We also quote \cite{MR4104476} where Molica Bisci and Vilasi obtained an existence result regarding positive solutions which are invariant under the action of a specific family of isometries and \cite{MR4124140} where Molica Bisci and Repov\v{s} showed the existence of positive solutions when the nonlinear term is critical in the sense of Sobolev. It is also worth mentioning \cite{MR4144286} where Cencelj et al. by applying the Palais principle of symmetric criticality and suitable group theoretical arguments are able to prove the existence of nontrivial weak solutions. Motived by the great interest in this field, in this paper we are going to study the Schr\"odinger equation on a non compact homogeneous Cartan-Hadamard manifold with a non linear term $f$ that oscillates near zero or at infinity. As regards oscillating nonlinearities there is a wide literature dealing with this kind of problems with numerous differential operator. To the best of our knowledge, one of the first contribution in this direction was given in \cite{MR1325805} by Habets et al. where the authors exhibit the problem they are considering admits an unbounded sequence of solutions with $d=1$ with a technique based on phase-plane analysis and time-mapping estimates. At a later time, Omari and Zanolin in \cite{MR1391521} were able to show the existence of infinitely many solutions for a problem with a general operator in divergence form building a sequence of arbitrarily large negative lower solutions and a sequence of arbitrarily large positive upper solutions. More recently Anello and Cordaro in \cite{MR1912413} proved the existence of a sequence of critical points converging to zero with respect the $L^\infty$ norm for a problem with a nonlinear oscillating term at zero. In the same spirit of the previous one, Molica Bisci and Pizzimenti obtained in \cite{MR3251760} similar results for the $p$-Kirchhoff problem analysing also what happens in presence of oscillations at infinity. Finally, Molica Bisci and R\u{a}dulescu in \cite{MR4149282} showed the existence of a sequence of invariant solutions tending to zero both in the Sobolev norm and in the $L^\infty$ norm on the Poincaré ball model.
One of the main task we have to face in order to study Problem \ref{eqP} is the loss of compactness of the embedding $H^1_g(\cM) \hookrightarrow L^q(\cM)$ due to the noncompactness of the manifold $\cM$. In order to overcome this difficulty, we will use an embedding result for a Sobolev space which is invariant under the action of a certain group proved by Skrzypczak and Tintarev in \cite{MR3101775} generalizing the well known fact that the embedding $H^1_{r}(\R^d) \hookrightarrow L^q(\R^d)$ is compact for all $q \in (2,2^*)$ for functions invariant under the group of the rotations. Coupling this fact fact with the principle of symmetric criticality proved by Palais in \cite{MR547524} and the continuity of the superposition operator, whose validity is established in \cite{MR546508} for the euclidean case and generalized to manifold in \cite[Proposition 2.5]{zbMATH01447265}, we will consider an auxiliary problem with a truncated nonlinearity and we will show the existence of infinitely many local minima. We emphasize that in dealing with the case of oscillations near zero we will assume no growth condition on the nonlinear term $f$. Our paper is organised as follows. At the end of this section we collect our main results. In section 2 we recall some basic concepts of Riemannian geometry and Sobolev spaces of manifold. In section 3 we prove Theorem \ref{th1} showing the existence of infinitely many critical points for the energy functional associated to \ref{eqP} and with both $L^\infty$ and Sobolev norm going to zero. In section 4 we address the problem of oscillations at infinity proving Theorem \ref{th2}. More precisely, given a group $G$ that acts on $\cM$ we will denote with
\begin{displaymath}
\fix(G):= \left\lbrace \sigma \in \cM \mid \varphi(\sigma)=\sigma \ \mbox{for all} \ \varphi \in G \right\rbrace.
\end{displaymath}
the fixed points of $G$. The following hypothesis will be crucial in the sequel:
\begin{description}
\item[$(\mathcal{H}^{\sigma_0}_{G})$] $G$ is a compact, connected subgroup of the isometries $\isom_g(\cM)$ of $(\cM,g)$ such that
	\begin{displaymath}
		\fix(G)=\left\lbrace \sigma_0\right\rbrace
	\end{displaymath}
for some point $\sigma_0 \in \cM$.
\end{description}

The main results we are going to prove during the rest of the paper are the following.
\begin{theorem} \label{th1}
Assume that $(\mathcal{H}^{\sigma_0}_{G})$ holds and let $\alpha \in L^1(\cM)\cap L^\infty(\cM) \setminus \lbrace 0 \rbrace$ be a nonnegative map such that $\alpha(\sigma)= \alpha(d_g(\sigma_0,\sigma))$. Moreover, let $f\colon \R \to \R$ be a continuous function for which 
\begin{description}
\item[$(f_0)$] there exist two sequences $(t_j)_j$ and $(t_j')_j$ with $\displaystyle\lim_{j \to +\infty} t'_j=0$ and $0 \leq t_j < t'_j$ such that
\begin{displaymath}
F(t_j)= \sup_{t \in \left[t_j,t_j' \right]}F(t),
\end{displaymath}
where $F(t):=\displaystyle\int_0^t f(\tau)\, d\tau;$
\item[$(f_1)$] there exist a constant $K_1>0$ and a sequence $(\xi_j)_j \subset (0,+\infty)$ with $\displaystyle\lim_{j \to + \infty} \xi_j=0$ such that
\begin{displaymath}
	\lim_{j \to +\infty} \frac{F(\xi_j)}{\xi_j^2}=+\infty,
\end{displaymath}
and
\begin{displaymath}
	\inf_{t \in \left[0,\xi_j \right]} F(t) \geq -K_1F(\xi_j).
\end{displaymath}
\end{description}
Then for every $\lambda >0$ it is possible to find a sequence $(w_j)_j \subset H^1_{G}(\cM)$ of nonnegative and not identically zero solutions of \eqref{eqP} such that
\begin{displaymath}
\lim_{j \to +\infty} \Vert w_j \Vert=\lim_{j \to +\infty} \Vert w_j \Vert_{L^{\infty}(\cM)}=0.
\end{displaymath}
\end{theorem}
\begin{theorem} \label{th2}
Assume that $(\mathcal{H}^{\sigma_0}_{G})$ holds and let $\alpha \in L^1(\cM)\cap L^\infty(\cM) \setminus \lbrace 0 \rbrace$ be a nonnegative map such that $\alpha(\sigma)= \alpha(d_g(\sigma_0,\sigma))$. Moreover, let $f\colon \R \to \R$ be a continuous function for which 
\begin{description}
\item[$(f_0')$] there are a constant $K_2>0$ and $q \in (2,2^*-1)$ such that
\begin{displaymath}
|f(t)| \leq K_2(1+|t|^q);
\end{displaymath}
\item[$(f_1')$] there are two sequences $(t_j)_j$ and $(t_j')_j$ with $\displaystyle\lim_{j \to +\infty} t_j=+\infty$ and $0 \leq t_j < t_j'$ such that
\begin{displaymath}
F(t_j)= \sup_{t \in \left[t_j,t_j' \right]}F(t);
\end{displaymath}
\item[$(f_2')$] there is a constant $K_3>0$ and a sequence $(\xi_j)_j \subset (0,+\infty)$ with $\displaystyle\lim_{j \to + \infty} \xi_j=\infty$ such that
\begin{displaymath}
\lim_{j \to +\infty} \frac{F(\xi_j)}{\xi_j^2}=+\infty,
\end{displaymath}
and
\begin{displaymath}
\inf_{t \in \left[0,\xi_j \right]} F(t) \geq -K_3 F(\xi_j).
\end{displaymath}
\end{description}
 Then for every $\lambda >0$ it is possible to find a sequence $(w_j)_j \subset H^1_{G}(\cM)$ of nonnegative and not identically zero weak solutions of \eqref{eqP}.
\end{theorem}
\section{Abstract framework}
This section is devoted to recall some basic concepts of Riemannian geometry and to fix the notation. Let $(\cM,g)$ be a $d$-dimensional Riemannian manifold where $g$ is a $(2,0)$ positive definite tensor and $g_{ij}$ are its component. We will denote the tangent space of $\cM$ at a point $ \sigma \in \cM$ with $T_{\sigma} \cM$. We recall that if $f\colon \cM \to \cN$, where $\cN$ is a $d'$--manifold the differential of $Df_\sigma \colon T_\sigma \cM \to T_{G(\sigma)} \cN$ is defined as
\begin{displaymath}
Df_{\sigma}(v)(h):=v(h \circ f )
\end{displaymath}
for all $ h \in C^\infty(\cN)$ and $v \in T_\sigma \cM$. From the notion of differential, if $A$ is a covariant $k$-tensor field of $M$ we can define a covariant $k$-tensor field $G^*A$ on $\cM$ defined as
\begin{displaymath}
(f^*A)_\sigma (v_1,...,v_k)=A_{f(\sigma)}(Df_\sigma(v_1),\ldots,Df_\sigma(v_k))
\end{displaymath}
for $v_1,...,v_k \in T_\sigma \cM$ called the pullback of $A$ by $f$. If $\cN$ is endowed with a metric $\tilde{g}$ we will say that $f$ is an isometry if $f^*\tilde{g}=g$. If $f$ is an isometry it is straightforward to show that it preserves the scalar product, i.e.
\begin{displaymath}
\langle Df_\sigma (v_1),Df_\sigma (v_2) \rangle_{f(\sigma)}= \langle v_1,v_2 \rangle_{\sigma}
\end{displaymath}
for $v_1,v_2 \in T_\sigma \cM $ where $\langle \cdot, \cdot \rangle_\sigma= g_\sigma (\cdot,\cdot)$.
In the following the group of all isometries $\varphi \colon\cM \to \cM$ will be denoted with $\isom_g(\cM)$.
If $\cS \subset \cM$ we can define
\begin{displaymath}
\diam(\cS):= \sup \left\lbrace d_g(\sigma_1,\sigma_2) \mid \sigma_1,\sigma_2 \in \cS \right\rbrace
\end{displaymath}
where $d_g$ is the geodesic distance on $\cM$. Here $H^1_g(\cM)$ denotes the usual Sobolev space defined as the closure of $C^\infty (\cM)$ with respect to the norm
\begin{displaymath}
\Vert w \Vert:= \left( \int_\cM |\nabla_g w(\sigma)|^2\, dv_g+\int_\cM |w(\sigma)|^2\, dv_g \right)^{\frac{1}{2}}
\end{displaymath}
and $\nabla_g w$ is the covariant derivative and
\begin{displaymath}
dv_g:=\sqrt{\det(g)} dx_1 \wedge...\wedge dx_d
\end{displaymath}
is the Riemannian volume form expressed in local coordinates. Once one has defined $dv_g$ it is possible to notice that it induces a measure on $\cM$. Namely, if $\cS \subset \cM$ we have
\begin{displaymath}
\operatorname{Vol}_g(\cS):= \int_\cS \, dv_g.
\end{displaymath}
The Laplace-Beltrami operator is defined in local coordinates by
\begin{displaymath}
\Delta_g h:= \sum_{i,j} \frac{1}{\sqrt{\strut \det g}} \frac{\partial}{\partial x^i} \left( g^{ij}\sqrt{\det g} \frac{\partial h}{\partial x^j} \right).
\end{displaymath}
We point out that we have defined $\Delta_g$ with the ``analyst's sign convention'', so that $-\Delta_g$ coincides with $-\Delta$ in $\mathbb{R}^d$ with its flat metric. The $(1,3)$ Riemann tensor is denoted with $R$ and we also recall that the sectional curvature is defined as
\begin{displaymath}
\cK_\sigma(v_1,v_2):= \frac{\langle R(v_1,v_2)v_1,v_2\rangle_\sigma}{\langle v_1,v_1 \rangle_\sigma \langle v_2,v_2 \rangle_\sigma-\langle v_1,v_2 \rangle_\sigma^2}
\end{displaymath}
for all $v_1, v_2\in T_\sigma \cM$. A Cartan-Hadamard manifold is a Riemannian Manifold that is complete and simply connected and has everywhere non-positive sectional curvature. We also say that a Riemannian manifold $\cM$ is homogeneous if for all $\sigma_1, \ \sigma_2 \in \cM$ there is a isometry $\varphi \in \isom_g(\cM)$ such that $\varphi(\sigma_1)= \sigma_2$. We will assume throughout the paper that the reader is familiar with some basic results on Riemannian geometry and Sobolev spaces on manifolds and we remind to the classical \cite{MR1138207, MR1481970, zbMATH01225960, zbMATH01447265} and \cite{MR2954043} for a deeper insight on these topics.
\begin{definition}
A group $G$ acting continuously on $\cM$ is said to be coercive if for every $t>0$ the set
\begin{displaymath}
	\left\lbrace \sigma \in \cM \mid \diam G\sigma \leq t \right\rbrace
\end{displaymath}
is bounded, where
\[
G\sigma:= \left\lbrace \varphi \cdot \sigma \mid \varphi \in G \right\rbrace.
\]

%\todo[inline]{(Simone) Dobbiamo definire l'insieme $G \sigma$.}
%where
%\begin{displaymath}
%O_t:= \left\lbrace \sigma \in \cM \mid \diam G\sigma \leq t \right\rbrace.
%\end{displaymath}
\end{definition}
As we will see later being coercive will play a determining role to have compact embedding for Sobolev spaces invariant under the action of a group $G$. Despite the coerciveness of a group $G$ has a clear geometrical meaning, it is a property that in most cases turns out to be difficult to verify. In order to overcome this problem, we introduce a condition that is equivalent in a Cartan-Hadamard manifold.

As pointed out in \cite[Proposition 3.1]{MR3101775} in a simply-connected Riemannian manifold with non positive sectional curvature, a subgroup $G$ of $\isom_g(\cM)$ satisfies ($\mathcal{H}^{\sigma_0}_{G}$) if and only if it is coercive. For the sake of completeness we write down here the proposition omitting the proof.
\begin{proposition} \label{prop1}
Let $\cM$ be a simply connected complete Riemannian manifold,
and assume that the sectional curvature is non-positive. Let $G$ be a
compact, connected subgroup of $\isom_g(\cM$) that fixes some point $\sigma_0 \in \cM$. Then $G$ is
coercive if and only if $G$ has no other fixed point but $\sigma_0$.
\end{proposition}
There are several example present in literature of homogeneous Cartan-Hadamard manifold with a group acting transitively on it fixing only one point. For instance $\R^d$ equipped with the euclidean metric and the special orthogonal group $SO(d)$ or  $SO(d_1) \times ... \times SO(d_h)$ where $\sum_{i=1}^{d_h} d_i=d$. Another common example is the Poincaré model $\mathbb{H}^d:=\left\lbrace x \in \R^d \ : \ |x| < 1 \right\rbrace$ endowed with the metric $$g_{ij}(x):=\frac{4}{(1-|x|^2)^2}\delta_{ij}$$ with the same choices as above for the group. In addition to that, we can also consider the set $P(d,\R)$ of the symmetric positive definite matrices with determinant equal to one. It turns out that it has a structure of homogeneous Cartan-Hadamard manifold and that the special orthogonal group $O(d)$ acts transitively on it fixing the identity matrix $I_d$. For further detail we suggest the reader to consult \cite[Chapter II.10]{MR1744486},  \cite{MR3490853}, \cite{MR3926120} and \cite[Chapter XXII]{MR1666820}.

Now we fix a point $\sigma_0 \in \cM$ and a group $G$ satisfying ($\mathcal{H}^{\sigma_0}_{G}$). We consider the Sobolev space
\begin{displaymath}
H^1_{G}(\cM)=\left\lbrace w \in H^1_g(\cM) \mid \varphi \circledast w=w \ \mbox{for all} \ \varphi \in G \right\rbrace
\end{displaymath}
where
\begin{displaymath}
\varphi \circledast w:= w(\varphi^{-1}\cdot\sigma) \quad \mbox{for a.e.} \ \sigma \in \cM.
\end{displaymath}

In virtue of the previous remark we are able to state the following compactness result.
\begin{lemma} \label{lemma1}
If $G$ satisfies $(\mathcal{H}^{\sigma_0}_{G})$, then the embedding
\begin{displaymath}
	H^1_{G}(\cM) \hookrightarrow L^{\nu}(\cM)
\end{displaymath}
is compact for all $\nu \in (2,2^*)$ where $2^*:=2d/(d-2)$.
\end{lemma}
\begin{proof}
According to \cite[Lemma 8.1 and Theorem 8.3]{zbMATH01447265} or \cite{MR365424} the embedding $H^1_{G}(\cM) \hookrightarrow L^{\nu}(\cM)$ is continuous for all $\nu \in \left[2,2^*\right]$ and cocompact for \cite[Chapter 9]{MR2294665}. At this point, taking into account Proposition \ref{prop1} we can apply
\cite[Theorem 1.3]{MR3101775} to complete the proof.
\end{proof}

\section{Oscillation at the origin}
In this section we investigate the existence of solutions for problem \eqref{eqP}
\begin{equation*} \label{eqP0}
\begin{cases}  \displaystyle -\Delta_g w+w=\lambda \alpha(\sigma) f(w) &\hbox{in}\ \cM \\
  w \in H^1_g(\cM)
\end{cases}
\end{equation*}
where $f$ represents a continuous function that oscillates near $0$. More precisely, since now till the end of the section the function $f$ satisfies hypothesis $(f_0)$ and $(f_1)$ of Theorem \ref{th1}.
As an immediate consequence of these hypothesis we have the following Lemma.
\begin{lemma} \label{lemma2}
If $f\colon \R \to \R$ is continuous and satisfies $(f_0)$ and $(f_1)$, then $f(0)=0$.
\end{lemma}
\begin{proof}
We first notice that
\begin{displaymath}
f(t_j)=\lim_{h \to 0^+}\frac{\displaystyle\int_{t_j}^{t_j+h} f(\tau)\, d\tau}{h}=\lim_{h \to 0^+}\frac{F(t_j+h)-F(t_j)}{h} \leq 0
\end{displaymath}
by ($f_0$). Thus, exploiting the continuity of $f$ we have \begin{displaymath}
	f(0)=\lim_{j \to + \infty}f(t_j)\leq 0.
\end{displaymath}
On the other hand, suppose by contradiction that $f(0)<0$. Then, again the continuity of $f$ implies that $f(t)<0$ for all $t \in \left[0,\delta \right)$ for some $\delta>0$. Then, we would have
\begin{displaymath}
\lim_{j \to + \infty} \frac{F(\xi_j)}{\xi_j^2} \leq 0,
\end{displaymath}
in contradiction with $(f_1)$.
\end{proof}

The relation $\alpha(\sigma)= \alpha(d_g(\sigma_0,\sigma))$ is a symmetry condition which replaces the radial symmetry of $\alpha$ is $\mathbb{R}^d$.
\begin{proof}[Proof of Theorem \ref{th1}]
%We fix $\lambda >0$ and we observe that for all}$t_0>0$ the is a constant $\kappa>0$ such that
%\begin{displaymath}
%|f(t)| \leq \kappa
%\end{displaymath}
%for all $t \in \left[0,t_0 \right]$. At this point we notice that it is not restrictive to suppose $\max \left\lbrace t_j, \xi_j  \right\rbrace \leq t_0$.
Let $\lambda>0$. Since $t_j \to 0$ and $\xi_j \to 0$ as $j \to +\infty$, we may assume that $0 \leq t_j \leq t_0$ and $0 \leq \xi_j \leq t_0$ for some $t_0>0$ and for every $j$. Let $\kappa = \max \left\lbrace \vert f(t) \vert \mid t \in [0,t_0] \right\rbrace$.
In view of Lemma \ref{lemma2}, we define the continuous truncated function
\begin{displaymath}
h(t):=
\begin{cases}
f(t_0) & \hbox{if $t >t_0$} \\
f(t) & \hbox{if $0 \leq t \leq t_0$} \\
0 & \hbox{if $t<0$}
\end{cases}
\end{displaymath}
and we consider the auxiliary problem
\begin{equation*} \label{eqP2}
\tag{$P_0$}
\begin{cases}  \displaystyle -\Delta_g w+w=\lambda \alpha(\sigma) h(w) &\hbox{in}\ \cM \\
  w \in H^1_{G}(\cM)
\end{cases}
\end{equation*}
We also set the the energy functional associated to Problem \eqref{eqP2}
\begin{displaymath}
J_{G,\lambda}(w):= \frac{1}{2}\Vert w \Vert^2-\lambda \int_\cM \alpha(\sigma)\left(\int_0^{w(\sigma)}h(\tau)\, d\tau\right)\, dv_g
\end{displaymath}
and we emphasize that $J_{G,\lambda} \in C^1(H^1_{G}(\cM),\R)$ thanks to Lemma \ref{lemma1} and that is is sequentially lower semicontinuous. Now, for all $j \in \N$ we define the set
\begin{displaymath}
\mathbb{E}_j^{G}:= \left\lbrace w \in H^1_{G}(\cM) \mid 0 \leq w(\sigma) \leq t_j' \ \mbox{a.e in} \ \cM \right\rbrace.
\end{displaymath}
We divide the remaining part of the proof in 6 steps.

\textbf{Step 1:} the functional $J_{G,\lambda}$ in bounded from below on $\mathbb{E}_j^{G}$ and attains its infimum on $\mathbb{E}_j^{G}$ at a function $u_j^{G} \in \mathbb{E}_j^{G}$.
Clearly for all $w \in \mathbb{E}_j^{G}$
\begin{align*}
\int_\cM \alpha(\sigma)\left(\int_0^{w(\sigma)}h(\tau)\, d\tau\right)\, dv_g &\leq  \int_\cM \alpha(\sigma)\left|\int_0^{w(\sigma)}h(\tau)\, d\tau\right|\, dv_g \\
& \leq \kappa \int_\cM \alpha(\sigma) w(\sigma) \, dv_g \leq \kappa \Vert \alpha \Vert_{L^1{(\cM)}} t_j'
\end{align*}
and so
\begin{equation} \label{eq9}
J_{G,\lambda}(w) \geq -\kappa \Vert \alpha \Vert_{L^1{(\cM)}} t_j'.
\end{equation}
At this point set
\begin{displaymath}
\iota_j^{G}:=\inf_{w \in \mathbb{E}_j^{G}} J_{G,\lambda}(w).
\end{displaymath}
From the definition of infimum, for all $k \in \N$ we can find $w_k \in \mathbb{E}_j^{G}$ such that
\begin{displaymath}
\iota_j^{G} \leq J_{G,\lambda}(w_k) \leq \iota_j^{G} + \frac{1}{k}.
\end{displaymath}
From this it follows
\begin{align*}
\Vert w_k \Vert^2 &=J_{G,\lambda}(w_k) + \lambda \int_\cM \alpha(\sigma)\left(\int_0^{w_k(\sigma)}h(\tau)\, d\tau\right)\, dv_g \\
& \leq \kappa \Vert \alpha \Vert_{L^1({\cM})} t_j' + \iota_j^{G} +1
\end{align*}
which implies $(w_k)_k$ must be bounded in $H^1_{G}(\cM)$. Then, up to a subsequence, we can assume $w_k \rightharpoonup u_j^{G}$ for some $u_j^{G} \in H^1_{G}(\cM)$. In order to prove that $u_j^{G} \in \mathbb{E}_j^{G}$ it sufficient to notice that the set $\mathbb{E}_j^{G}$ is closed and convex, thus weakly closed. Now, exploiting the sequentially lower semicontinuity of $J_{G,\lambda}$ we get
\begin{displaymath}
\iota_j^{G} \leq J_{G,\lambda} (u_j^{G}) \leq \liminf_{k \to \infty} J_{G,\lambda} (w_k) \leq \iota_j^{G}
\end{displaymath}
hence
\begin{displaymath}
\iota_j^{G}=J_{G,\lambda} (u_j^{G}).
\end{displaymath}
\textbf{Step 2:} for all $j \in \N$ one has that $0 \leq u_j^{G}(\sigma) \leq t_j$ a.e. in $\cM$.

In order to show that, we set the Lipschitz continuous function $\varrho_j\colon \R \to \R$
\begin{displaymath}
\varrho_j(t):=
\begin{cases}
t_j & \mbox{if $t > t_j$}\\
t & \mbox{if $0 \leq t \leq t_j$}\\
0 & \mbox{if $t < 0$}
\end{cases}
\end{displaymath}
we can consider the superposition operator $T_j \colon H^1_g(\cM) \to H^1_g(\cM)$ defined as
\begin{displaymath}
T_jw(\sigma):=\varrho_j(w(\sigma)) \quad \mbox{a.e. in} \ \cM.
\end{displaymath}
From \cite[Proposition 2.5]{zbMATH01447265} it follows that $T_j$ is a continuous operator. Furthermore, if we restrict $T_j$ to the $G$-invariant functions we have $T_j\colon H^1_{G}(\cM) \to H^1_{G}(\cM)$. In fact, one can readily see that
\begin{align*}
\varphi\circledast T_jw(\sigma) & = T_jw(\varphi^{-1}\cdot \sigma)=(\varrho_j \circ w)(\varphi^{-1}\cdot \sigma) \\
& = \varrho_j(w(\varphi^{-1}\cdot \sigma))=\varrho_j(w(\sigma))= (\varrho_j \circ w)(\sigma)\\
&=T_jw(\sigma) \quad \mbox{a.e. in} \ \cM
\end{align*}
for all $w \in H^1_{G}(\cM)$ and $\varphi\in G$. Actually from its definition  it is clear that $T_jw \in \mathbb{E}_j^{G}$ for all $j \in \N$. At this point we set $v_{G,j}^{\star}:= T_j u_j^{G}$ and
\begin{displaymath}
X_j^{G}:= \left\lbrace \sigma \in \cM \ : \ t_j < u_j^{G}(\sigma) \leq t_j' \ \right\rbrace.
\end{displaymath}
Observe that for all $ \sigma \in X_j^{G}$ one has
\begin{displaymath}
v_{G,j}^{\star}(\sigma)=T_ju_j^{G}(\sigma)=t_j.
\end{displaymath}
Now, exploiting $(f_0)$ we get
\begin{equation*}
\int_0^{u_j^{G}(\sigma)} h(\tau)\, d\tau \leq \sup_{t \in \left[ t_j,t_j' \right]}\int_0^t h(\tau)\, d\tau =\int_0^{t_j} h(\tau)\, d\tau=\int_0^{v_{G,j}^{\star}(\sigma)} h(\tau)\, d\tau
\end{equation*}
thus
\begin{equation} \label{eq1}
\int^{v_{G,j}^{\star}(\sigma)}_{u_j^{G}(\sigma)} h(\tau)\,d\tau \geq 0
\end{equation}
for all $\sigma \in X_j^{G}$. Moreover, taking into account the fact that $|\nabla_g v_{G,j}^{\star} (\sigma)|=0$ a.e. in $X_j^{G}$, we obtain
\begin{align} \label{eq2}
\Vert v_{G,j}^{\star} \Vert^2-\Vert u_j^{G} \Vert^2&= \int_\cM \left(|\nabla_g v_{G,j}^{\star}(\sigma)|^2 - |\nabla_g u_j^{G}(\sigma)|^2\right) \, dv_g\notag \\
&{}\quad +\int_\cM \left(|v_{G,j}^{\star}(\sigma)|^2-|u_j^{G}(\sigma)|^2\right) \, dv_g \notag \\
& = -\int_{X_j^{G}} |\nabla_g u_j^{G}(\sigma)|^2 \, dv_g+\int_{X_j^{G}} \left(t_j^2-|u_j^{G}(\sigma)|^2\right) \, dv_g \\
& \leq -\int_{X_j^{G}} |\nabla_g v_{G,j}^{\star} (\sigma)-\nabla_g u_j^{G}(\sigma)|^2 \, dv_g-\int_{X_j^{G}} \left|u_j^{G}(\sigma)-t_j\right|^2 \, dv_g\notag \\
& = -\int_{\cM} |\nabla_g v_{G,j}^{\star} (\sigma)-\nabla_g u_j^{G}(\sigma)|^2 \, dv_g-\int_{\cM} \left|u_j^{G}(\sigma)-v_{G,j}^{\star}(\sigma)\right|^2 \, dv_g \notag \\
&= -\Vert v_{G,j}^{\star}- u_j^{G} \Vert^2 \notag.
\end{align}
At this point, in virtue of \eqref{eq1} and \eqref{eq2}, recalling $v_{G,j}^{\star} \in \mathbb{E}_j^{G}$ we have
\begin{align*}
0 & \leq J_{G,\lambda}(v_{G,j}^{\star})-J_{G,\lambda}(u_j^{G})= \frac{\Vert v_{G,j}^{\star} \Vert^2-\Vert u_j^{G} \Vert^2}{2} -\lambda \int_{\cM} \alpha(\sigma)\left( \int^{v_{G,j}^{\star}(\sigma)}_{u_j^{G}(\sigma)} h(\tau)\,d\tau\right)\, dv_g \\
& \leq -\frac{1}{2}\Vert v_{G,j}^{\star} - u_j^{G} \Vert^2 -\lambda \int_{X_j^{G}} \alpha(\sigma)\left( \int^{v_{G,j}^{\star}(\sigma)}_{u_j^{G}(\sigma)} h(\tau)\,d\tau\right)\, dv_g \\
& \leq -\frac{1}{2}\Vert v_{G,j}^{\star} - u_j^{G} \Vert^2.
\end{align*}
From this, we can deduce
\begin{displaymath}
\Vert v_{G,j}^{\star} - u_j^{G} \Vert^2=0.
\end{displaymath}
Since $v_{G,j}^\star \neq u_j^G$ except on $X_j^G$, we deduce that
%which means
%\begin{equation} \label{eq3}
%\int_{X_j^{G}} |\nabla_g( v_{G,j}^{\star}-u_j^{G})(\sigma)|^2 \, dv_g+\int_{X_j^{G}} \left|(v_{G,j}^{\star}-u_j^{G})(\sigma)\right|^2 \, dv_g=0.
%\end{equation}
$\operatorname{Vol}_g(X_j^{G})=0$ as desired.

\textbf{Step 3:} the function $u_j^{G}$ is a local minimum for $J_{G,\lambda}$ in the Sobolev space $H^1_{G}(\cM)$ for all $j \in \N$.

In order to do that, we select $ w \in H^1_{G}(\cM)$ and we set
\begin{displaymath}
Z_j^{G}:=\left\lbrace \sigma \in \cM \ : \ w(\sigma) \notin \left[0,t_j \right] \right\rbrace
\end{displaymath}
for every $j \in \N$. Recalling the superposition operator defined in step 2 we set
\begin{displaymath}
v_j^\star(\sigma):=T_j w(\sigma)=
\begin{cases}
t_j & \mbox{if $w(\sigma) > t_j$} \\
w(\sigma) & \mbox{if $0 \leq w(\sigma) \leq t_j$} \\
0 & \mbox{if $w(\sigma) < 0$}
\end{cases}
\end{displaymath}
Now, on one hand one can easily see that
\begin{displaymath}
\int_{v_j^\star(\sigma)}^{w(\sigma)}h(\tau)\, d\tau=0
\end{displaymath}
for every $\sigma \in \cM \setminus Z_j^{G}$. On the other hand, if $ \sigma \in Z_j^{G}$ only three alternative can occur
\begin{enumerate}
\item If $w(\sigma) \leq 0$ it is immediate to see
\begin{displaymath}
\int_{v_j^\star(\sigma)}^{w(\sigma)}h(\tau)\, d\tau=\int_{0}^{w(\sigma)}h(\tau)\,d\tau=0.
\end{displaymath}
\item If $t_j<w(\sigma)\leq t_j'$ we have
\begin{align*}
\int_{v_j^\star(\sigma)}^{w(\sigma)}h(\tau)\,d\tau &= \int_{0}^{w(\sigma)}h(\tau)\,d\tau-\int^{v_j^\star(\sigma)}_{0}h(\tau)\,d\tau \\
& = \int_{0}^{w(\sigma)}h(\tau)\,d\tau-\int^{t_j}_{0}h(\tau)\,d\tau \\
& \leq  \int_{0}^{w(\sigma)}h(\tau)\,d\tau-\sup_{t \in \left[t_j,t_j' \right]}\int^{t}_{0}h(\tau)\,d\tau \leq 0.
\end{align*}
\item If $w(\sigma)>t_j'$ we obtain
\begin{align} \label{eq4}
\int_{v_j^\star(\sigma)}^{w(\sigma)}h(\tau)\,d\tau & =\int_{t_j}^{w(\sigma)}h(\tau)\,d\tau \\ \notag
& \leq \left| \int_{t_j}^{w(\sigma)}h(\tau)\,d\tau \right| \leq \kappa (w(\sigma)-t_j)
\end{align}
At this point set
\begin{displaymath}
C:=\kappa\Vert \alpha \Vert_{L^{\infty}(\cM)} \sup_{t \geq t_j'}\frac{t-t_j}{(t-t_j)^{\nu}}
\end{displaymath}
where $\nu \in (2,2^*)$. From this and \eqref{eq4} we have
\begin{align} \label{eq5}
\int_\cM \alpha(\sigma) \left(\int_{v_j^\star(\sigma)}^{w(\sigma)}h(\tau)\,d\tau \right) \, dv_g & \leq \Vert \alpha \Vert_{L^{\infty}(\cM)} \int_\cM \left(\int_{v_j^\star(\sigma)}^{w(\sigma)}h(\tau)\, d\tau \right) \, dv_g \\\nonumber
& \leq C \int_\cM (w(\sigma)-t_j)^\nu\, dv_g.
\end{align}
Denote
\begin{displaymath}
\gamma:= \sup_{w \in H^1_{G}(\cM)\setminus \left\lbrace 0 \right\rbrace} \frac{\Vert w \Vert_{L^{\nu}(\cM)}}{\Vert w \Vert}
\end{displaymath}
and observe that is finite for Lemma \ref{lemma1}, from \eqref{eq5} we deduce
\begin{equation} \label{eq7}
\int_\cM \alpha(\sigma) \left(\int_{v_j^\star(\sigma)}^{w(\sigma)}h(\tau)\,d\tau \right) \, dv_g \leq C \gamma^\nu \Vert w-v_j^\star\Vert^\nu.
\end{equation}
\end{enumerate}
Now, we compute
\begin{align}\label{eq6}
\Vert w \Vert^2-\Vert v_{j}^{\star} \Vert^2&= \int_\cM \left(|\nabla_g w(\sigma)|^2 - |\nabla_g v_{j}^{\star}(\sigma)|^2\right) \, dv_g+\int_\cM \left(|w(\sigma)|^2-|v_{j}^{\star}|^2\right) \, dv_g \notag \\
& =\int_{Z_j^{G}} |\nabla_g w(\sigma)|^2  \, dv_g+\int_{Z_j^{G,-}} |w(\sigma)|^2 \, dv_g +\int_{Z_j^{G,+}} |w(\sigma)-t_j|^2 \, dv_g \notag \\
&=\int_{Z_j^{G}} |\nabla_g w(\sigma)-\nabla_g v_{j}^{\star}(\sigma)|^2  \, dv_g+\int_{Z_j^{G,-}} |w(\sigma)- v_{j}^{\star}(\sigma)|^2 \, dv_g  \\
& \quad +\int_{Z_j^{G,+}} |w(\sigma)-t_j|^2 \, dv_g \notag \\
&=\Vert w-v_j^{\star} \Vert^2 \notag
\end{align}
where
\begin{displaymath}
Z_j^{G,+}:= \left\lbrace \sigma \in Z_j^{G} \ : \ w(\sigma) >0\right\rbrace \quad \mbox{and} \quad Z_j^{G,-}:=\left\lbrace \sigma \in Z_j^{G} \ : \ w(\sigma) <0\right\rbrace
\end{displaymath}
Coupling \eqref{eq7} and \eqref{eq6} we get
\begin{align*}
 J_{G,\lambda}(w)-J_{G,\lambda}(v_j^\star)&= \frac{\Vert w \Vert^2-\Vert v_{j}^{\star} \Vert^2}{2}-\lambda \int_\cM \alpha(\sigma) \left(\int_{v_j^\star(\sigma)}^{w(\sigma)}h(\tau)\,d\tau \right) \, dv_g \\
& \geq \frac{1}{2}\Vert w-v_j^{\star} \Vert^2-\lambda C \gamma^\nu \Vert w-v_j^\star\Vert^\nu.
\end{align*}
In view oh that, recalling $J_{G,\lambda}(v_j^\star) \geq J_{G,\lambda}(u_j^{G})$ since $ v_j^\star \in \mathbb{E}_j^{G}$, we obtain
\begin{equation} \label{eq8}
J_{G,\lambda}(w) \geq J_{G,\lambda}(u_j^{G})+ \Vert w-v_j^{\star} \Vert^2 \left(\frac{1}{2}-\lambda C \gamma^\nu \Vert w-v_j^{\star} \Vert^{\nu-2} \right)
\end{equation}
At this point, we notice that
\begin{displaymath}
\Vert w-v_j^\star \Vert \leq \Vert w-u_j^{G} \Vert + \Vert u_j^{G}-v_j^\star \Vert = \Vert w-u_j^{G} \Vert + \Vert T_j(u_j^{G}-w) \Vert
\end{displaymath}
thus, exploiting the continuity of the superposition operator, it is possible to find a $\delta >0$ such that
\begin{displaymath}
\Vert w-v_j^\star \Vert^{\nu-2} \leq \frac{1}{4 \lambda C \gamma^\nu}
\end{displaymath}
if $\Vert w-u_j^{G} \Vert \leq \delta$.
Hence, from \eqref{eq8} we get
\begin{displaymath}
J_{G,\lambda}(w) \geq J_{G,\lambda}(u_j^{G})
\end{displaymath}
that means $u_j^{G}$ is a local minimizer.

\textbf{Step 4:} If
\begin{displaymath}
\iota_j^{G}:=\inf_{w \in \mathbb{E}_j^{G}} J_{G,\lambda}(w).
\end{displaymath}
then
\begin{displaymath}
\lim_{j \to \infty}\iota_j^{G}=\lim_{j \to \infty}\Vert u_j^{G}\Vert=0
\end{displaymath}
Recalling that $u_j^{G} \in \mathbb{E}_j^{G}$ and that $\iota_j^{G}=J_{G,\lambda}(u_j^{G})$ we have
\begin{align} \label{eq10}
\int_\cM |\nabla_g u_j^{G} (\sigma)|^2\, dv_g+\int_\cM |u_j^{G}(\sigma)|^2\, dv_g& = J_{G,\lambda}(u_j^{G}) + \lambda \int_\cM \alpha(\sigma)\left(\int_0^{u_j^{G}(\sigma)} h(\tau)\, d\tau \right)\,dv_g \notag \\
&=\iota_j^{G} +\lambda \int_\cM \alpha(\sigma)\left(\int_0^{u_j^{G}(\sigma)} h(\tau)\, d\tau \right)\,dv_g \\
& \leq \iota_j^{G} + \lambda \kappa \Vert \alpha \Vert_{L^1{(\cM)}} t_j' \notag
\end{align}
At this point, we notice that the function $w_0=0$ belongs to $\mathbb{E}_j^{G}$ and so
\begin{displaymath}
\iota_j^{G} = \inf_{w \in \mathbb{E}_j^{G}} J_{G,\lambda}(w) \leq 0.
\end{displaymath}
From this and \eqref{eq10} we can deduce
\begin{displaymath}
\lim_{j \to \infty}\Vert u_j^{G}\Vert=0
\end{displaymath}
since $t_j' \to 0$ as $j \to \infty$. Furthermore, recalling \eqref{eq9} we obtain
\begin{displaymath}
-\kappa \Vert \alpha \Vert_{L^1{(\cM)}} t_j' \leq \iota_j^{G} \leq 0
\end{displaymath}
from which is clear that
\begin{displaymath}
\lim_{j \to \infty} \iota_j^{G}=0.
\end{displaymath}

\textbf{Step 5:} for all $ j \in \N$ we have
\begin{displaymath}
\iota_j^{G} <0.
\end{displaymath}
In order to do that, we select $j \in \N$ and $0<a<b$ such that
\begin{equation} \label{eq11}
\essinf_{\sigma \in A_a^b} \alpha(\sigma)\geq \alpha_0 >0
\end{equation}
where
\begin{displaymath}
A_a^b=B_{\sigma_0}(a+b)\setminus B_{\sigma_0}(b-a)
\end{displaymath}
and, after fixing $\varepsilon \in (0,1)$ we define the function
\begin{displaymath}
\vartheta_{a,b}^{\varepsilon}(\sigma):=
\begin{cases}
0 & \mbox{if $\sigma \in \cM \setminus A_a^b$} \\
1 & \mbox{if $\sigma \in A_{\varepsilon a}^b$} \\
\dfrac{a-|d_g(\sigma_0,\sigma)-b|}{(1-\varepsilon)a} & \mbox{if $\sigma \in A_a^b \setminus A_{\varepsilon a}^b$}.
\end{cases}
\end{displaymath}
It is straightforward to verify that $ \vartheta_{a,b}^{\varepsilon} \in H^1_{G}(\cM)$ since in each point its value depends only on the distance from $\sigma_0$. Moreover, one can easily verifies that
$\supp(\vartheta_{a,b}^{\varepsilon}) \subset A_a^b$ and $  \Vert \vartheta_{a,b}^{\varepsilon} \Vert_{L^\infty(\cM)} \leq 1$. At this point we define the map $\mu_g \colon (0,1) \to \R$ a
where
\begin{displaymath}
\mu_g (\varepsilon)= \frac{\operatorname{Vol}_g(A_{\varepsilon a}^b)}{\operatorname{Vol}_g(A_a^b \setminus A_{\varepsilon a}^b)}
\end{displaymath}
and we notice that
\begin{displaymath}
\lim_{\varepsilon \to 0^+} \mu_g(\varepsilon)=0, \quad \quad \quad \lim_{\varepsilon \to 1^-} \mu_g(\varepsilon)=+\infty.
\end{displaymath}
In view of that, it is possible to find $\varepsilon_0 \in (0,1)$ such that
\begin{displaymath}
\frac{\operatorname{Vol}_g(A_{\varepsilon_0 a}^b)}{\operatorname{Vol}_g(A_a^b \setminus A_{\varepsilon_0 a}^b)}=K_1+1
\end{displaymath}
where $K_1>0$ is the constant given in hypothesis $(f_1)$. From $(f_1)$ we also have the existence of an index $k_0$, with $\xi_{k_0} \leq t_j'$ such that for every $k \geq k_0$
\begin{displaymath}
\frac{\displaystyle\int_0^{\xi_k}h(\tau)\, d\tau}{\xi_k^2} > \frac{1}{2\lambda} \frac{(K+1)\Vert  \vartheta_{a,b}^{\varepsilon_0}\Vert^2}{\alpha_0 \operatorname{Vol}_g(A_{\varepsilon_0a}^b)}.
\end{displaymath}
From this, $(f_1)$ and \eqref{eq11} it follows
\allowdisplaybreaks
\begin{multline*}
\frac{\displaystyle\int_{A_a^b} \alpha(\sigma)\left(\int_0^{\xi_k \vartheta_{a,b}^{\varepsilon_0}(\sigma)}h(\tau)\, d\tau\right)\, dv_g}{\Vert \xi_k \vartheta_{a,b}^{\varepsilon_0}\Vert^2} = \\
= \frac{\displaystyle\int_{A_{\varepsilon_0 a}^b} \alpha(\sigma)\left(\int_0^{\xi_k }h(\tau)\, d\tau\right)\, dv_g}{\xi_k^2\Vert  \vartheta_{a,b}^{\varepsilon_0}\Vert^2}
+\frac{\displaystyle\int_{A_a^b \setminus A_{\varepsilon_0 a}^b} \alpha(\sigma)\left(\int_0^{\xi_k \vartheta_{a,b}^{\varepsilon_0}(\sigma)}h(\tau)\, d\tau\right)\, dv_g}{\xi_k^2\Vert  \vartheta_{a,b}^{\varepsilon_0}\Vert^2} \\
 \geq \alpha_0 \frac{\displaystyle\int_{A_{\varepsilon_0 a}^b} \left(\int_0^{\xi_k }h(\tau)\, d\tau\right)\, dv_g}{\xi_k^2\Vert  \vartheta_{a,b}^{\varepsilon_0}\Vert^2}
+\alpha_0 \frac{\displaystyle\int_{A_a^b \setminus A_{\varepsilon_0 a}^b} \left(\inf_{t \in \left[0,\xi_k \right]}\int_0^{t}h(\tau)\, d\tau\right)\, dv_g}{\xi_k^2\Vert  \vartheta_{a,b}^{\varepsilon_0}\Vert^2} \\
\!\! \!\! \!\! \geq \alpha_0 \frac{\displaystyle\int_{A_{\varepsilon_0 a}^b} \left(\displaystyle\int_0^{\xi_k }h(\tau)\, d\tau\right)\, dv_g}{\xi_k^2\Vert  \vartheta_{a,b}^{\varepsilon_0}\Vert^2}
-K_1\alpha_0 \frac{\displaystyle\int_{A_a^b \setminus A_{\varepsilon_0 a}^b} \left(\int_0^{\xi_k}h(\tau)\, d\tau\right)\, dv_g}{\xi_k^2\Vert  \vartheta_{a,b}^{\varepsilon_0}\Vert^2} \\
= \alpha_0\frac{\operatorname{Vol}_g(A_{\varepsilon_0a}^b)}{(K_1+1)\Vert  \vartheta_{a,b}^{\varepsilon_0}\Vert^2} \frac{\displaystyle\int_0^{\xi_k}h(\tau)\, d\tau}{\xi_k^2} > \frac{1}{2\lambda}
\end{multline*}
for all $k \geq k_0$.
Now, from the definition of $\xi_k \vartheta_{a,b}^{\varepsilon}$ it is clear that $\xi_k \vartheta_{a,b}^{\varepsilon_0} \in \mathbb{E}_j^{G}$. Hence $J_{G,\lambda}(\xi_k \vartheta_{a,b}^{\varepsilon_0}) <0$ and as a consequence of that $\iota_j^{G} <0$ as desired.

\textbf{Step 6:} the function $u_j^{G}$ is a local minimum for the functional $J_{G,\lambda}$ in the Sobolev space $H^1_g(\cM)$ for all $j \in \N$.

Since $\Vert u_j^{G}\Vert_{L^\infty(\cM)} \to 0$ as $j \to \infty$, up to relabel the indexes we can assume the existence of a sequence $(u_j^{G})_j \subset H^1_g(\cM)$ such that
\begin{equation} \label{eq16}
\Vert u_j^{G}\Vert_{L^\infty(\cM)} \leq t_0.
\end{equation}
At this point, in virtue of the Principle of Symmetric Criticality of Palais (see \cite{MR547524} for details), to conclude the proof it is sufficient to show that $J_{G,\lambda}$ is invariant under the action of $G$. Consider first $\Vert \cdot \Vert$. For all $\varphi \in G$ and $w \in H^1_g(\cM)$ we have
\begin{align} \label{eq12}
\Vert \varphi \circledast w \Vert^2&= \int_\cM |\nabla_g ( \varphi \circledast w)(\sigma)|^2\, dv_g+\int_\cM |(\varphi \circledast w)(\sigma)|^2\, dv_g \notag \\
& =\int_\cM |\nabla_g (  w(\varphi^{-1} \cdot \sigma))|^2\, dv_g+\int_\cM | w(\varphi^{-1} \cdot\sigma)|^2\, dv_g  \notag \\
&= \int_\cM \langle D\varphi_{\varphi^{-1}\cdot \sigma}\nabla_g  w(\varphi^{-1} \cdot \sigma),D\varphi_{\varphi^{-1}\cdot \sigma}\nabla_g  w(\varphi^{-1} \cdot \sigma)\rangle_{\varphi^{-1} \cdot \sigma}\, dv_g \notag \\
&\quad {} +\int_\cM | w(\varphi^{-1} \cdot\sigma)|^2\, dv_g \notag \\
& = \int_\cM \langle \nabla_g  w(\varphi^{-1} \cdot \sigma),\nabla_g  w(\varphi^{-1} \cdot \sigma)\rangle_{\varphi^{-1} \cdot \sigma}\, dv_g+\int_\cM | w(\varphi^{-1} \cdot\sigma)|^2\, dv_g \\
&=\int_\cM \langle \nabla_g  w( \tilde{\sigma}),\nabla_g  w(\tilde{\sigma})\rangle_{\tilde{\sigma}}\, dv_{(\varphi^{-1})^*g}+\int_\cM | w(\tilde{\sigma})|^2\, dv_{(\varphi^{-1})^*g} \notag \\
&= \int_\cM \langle \nabla_g  w( \tilde{\sigma}),\nabla_g  w(\tilde{\sigma})\rangle_{\tilde{\sigma}}\, dv_g+\int_\cM | w(\tilde{\sigma})|^2\, dv_g= \Vert w \Vert^2 \notag
\end{align}
since $\varphi$ is an  isometry and preserves scalar products. Furthermore
\begin{displaymath}
\alpha(\varphi^{-1} \cdot \sigma)=\alpha(d_g(\sigma_0,\varphi^{-1}\sigma))= \alpha(d_g(\varphi^{-1} \cdot\sigma_0,\varphi^{-1}\sigma)) =\alpha (d_g(\sigma_0,\sigma))= \alpha(\sigma)
\end{displaymath}
which implies
\begin{align} \label{eq13}
\int_\cM \alpha(\sigma)\left(\int_0^{w(\varphi^{-1}\cdot \sigma)}h(\tau)\, d\tau\right)\, dv_g &= \int_\cM \alpha(\varphi^{-1}\cdot\sigma)\left(\int_0^{w(\varphi^{-1}\cdot \sigma)}h(\tau)\, d\tau\right)\, dv_g \notag\\
&= \int_\cM \alpha(\tilde{\sigma})\left(\int_0^{w(\tilde{ \sigma})}h(\tau)\, d\tau\right)\, dv_{(\varphi^{-1})^*g} \\
&=\int_\cM \alpha(\tilde{\sigma})\left(\int_0^{w(\tilde{ \sigma})}h(\tau)\, d\tau\right)\, dv_g \notag.
\end{align}
Putting together \eqref{eq12} and \eqref{eq13} we obtain
\begin{displaymath}
J_{G,\lambda}(\varphi \circledast w)=J_{G,\lambda}(w)
\end{displaymath}
hence, applying the Principle of Symmetric Criticality of Palais we have that each element of the sequence $u_j^{G}$ is a critical point of the functional $J_{G,\lambda}$ and a weak solution of \eqref{eqP2}. Furthermore, recalling Step 2 and \eqref{eq16} we also have that $u_j^{G}$ is a solution of our original problem \eqref{eqP}.
\end{proof}

\begin{example}
	Let $f \colon \mathbb{R} \to \mathbb{R}$ be defined by
\begin{displaymath}
f(t) :=
	\begin{cases}
		9 \sqrt{t} \displaystyle\sin \left( \frac{1}{\sqrt[3]{t}} \right) - 2 \sqrt[6]{t} \cos  \left( \frac{1}{\sqrt[3]{t}} \right) &\hbox{if $t \geq 0$} \\
0 &\hbox{if $t < 0$},
	\end{cases}
\end{displaymath}
whose primitive is
\begin{displaymath}
F(t) = \int_0^t f(s) \, ds =
	\begin{cases}
		\displaystyle 6 t^{3/2} \sin \left( \frac{1}{\sqrt[3]{t}} \right)  &\hbox{if $t \geq 0$} \\
0 &\hbox{if $t < 0$}.
	\end{cases}
\end{displaymath}
As in  \cite{MR1912413} one can check that conditions $(f_0)$ and $(f_1)$ are satisfied.
\end{example}

\section{Oscillations at infinity}
In this section we investigate the solutions of problem \eqref{eqP}
\begin{equation*} \label{eqPinf}
\begin{cases}  \displaystyle -\Delta_g w+w=\lambda \alpha(\sigma) f(w) &\hbox{in}\ \cM \\
  w \in H^1_g(\cM)
\end{cases}
\end{equation*}
where $f \colon \R \to \R$ is a continuous nonnegative function that oscillates at infinity. Preferring a variational approach, we define the energy functional $J_\lambda \colon H^1_g(\cM) \to \R$ associated to problem \eqref{eqP} where
\begin{displaymath}
J_\lambda(w):=\frac{1}{2}\Vert w \Vert^2-\lambda \int_\cM \alpha(\sigma)F(w(\sigma))\, dv_g,
\end{displaymath}
and $F(t):=\displaystyle\int_0^t f(\tau)\, d\tau$. As regard the right hand side of \eqref{eqP}, we make  on the nonlinear term $f$ the hypothesis $(f_0')$ and $(f_2')$ of Theorem \ref{th2} till the end of the section. As we already did in the previous section we will fist look for solutions for a truncated problem and then we will show that they also solves \eqref{eqP}. In order to do that, we start defining the function
\begin{displaymath}
h(t):=
\begin{cases}
f(t) & \mbox{if $t \geq 0$}  \\
f(0) & \mbox{if $t<0$}
\end{cases}
\end{displaymath}
and considering the auxiliary problem
\begin{equation} \label{eqP3}
\tag{$P_\infty$}
\begin{cases}  \displaystyle -\Delta_g w+w=\lambda \alpha(\sigma) h(w) &\hbox{in}\ \cM \\
  w \in H^1_{G}(\cM).
\end{cases}
\end{equation}
We associate to problem \eqref{eqP3} the functional
\begin{displaymath}
J_{G,\lambda}(w):= \frac{1}{2}\Vert w \Vert^2-\lambda \int_\cM \alpha(\sigma)\left(\int_0^{w(\sigma)}h(\tau)\, d\tau\right)\, dv_g
\end{displaymath}
and we point out that $J_{G,\lambda} \in C^1(H^1_{G}(\cM),\R)$ and again thanks to Lemma \ref{lemma1} that is sequentially lower semicontinuous. We emphasize that nonnegative critical points of $J_{G,\lambda}(w)$ are also critical point for the functional $J_\lambda$.

\begin{proof}[Proof of Theorem \ref{th2}]
Since some arguments of the proof are very similar the the ones described in Theorem \ref{th1} we will omit them. Fix $\lambda>0$.  We start for every $j \in \N$  setting
\begin{displaymath}
\mathbb{E}_j^{G}:= \left\lbrace w \in H^1_{G}(\cM) \ : \ 0 \leq w(\sigma) \leq t_j' \ \mbox{a.e in} \ \cM \right\rbrace.
\end{displaymath}
\textbf{Step 1:} the functional $J_{G,\lambda}$ in bounded from below on $\mathbb{E}_j^{G}$ and attains its infimum on $\mathbb{E}_j^{G}$ at a function $w_j^{G} \in \mathbb{E}_j^{G}$.

From hypothesis $(f_0')$ we obtain
\begin{displaymath}
\int_0^{w(\sigma)}h(\tau)\, d\tau \leq K_2 \left( t_j' + \frac{\left(t_j'\right)^{q+1}}{q+1} \right)
\end{displaymath}
As a consequence of that
\begin{displaymath}
J_{G,\lambda}(w) \geq - \lambda K_2 \Vert \alpha \Vert_{L^1(\cM)} \left( t_j' + \frac{\left(t_j'\right)^{q+1}}{q+1} \right)
\end{displaymath}
which imply that $J_{G,\lambda}$ is bounded from below on $\mathbb{E}_j^{G}$ for every $j\in \N$. At this point, follows the line of Step 1 in Theorem \ref{th1} we can find $u_j^{G}$ such that
\begin{displaymath}
\iota_j^{G}:=\inf_{w \in \mathbb{E}_j^{G}} J_{G,\lambda}(w)=J_{G,\lambda}(u_j^{G}).
\end{displaymath}
\textbf{Step 2:} for all $j \in \N$ one has that $0 \leq u_j^{G}(\sigma) \leq t_j$ a.e. in $\cM$.

The statement follows following closely the line of the proof of Step 2 on Theorem \ref{th1}.

\textbf{Step 3:} the function $u_j^{G}$ is a local minimum for $J_{G,\lambda}$ in the Sobolev space $H^1_{G}(\cM)$ for all $j \in \N$

In order to show that, we choose $ w \in H^1_{G}(\cM)$ and we set
\begin{displaymath}
Z_j^{G}:=\left\lbrace \sigma \in \cM \ : \ w(\sigma) \notin \left[0,t_j \right] \right\rbrace
\end{displaymath}
for every $j \in \N$. Recalling the superposition operator defined in step 2 of Theorem \ref{th1} we set
\begin{displaymath}
v_j^\star(\sigma):=T_j w(\sigma)=
\begin{cases}
t_j & \mbox{if $w(\sigma) > t_j$} \\
w(\sigma) & \mbox{if $0 \leq w(\sigma) \leq t_j$} \\
0 & \mbox{if $w(\sigma) < 0$}.
\end{cases}
\end{displaymath}
Now, on one hand one can easily see that
\begin{displaymath}
\int_{v_j^\star(\sigma)}^{w(\sigma)}h(\tau)\, d\tau=0
\end{displaymath}
for every $\sigma \in \cM \setminus Z_j^{G}$. On the other hand, if $ \sigma \in Z_j^{G}$ we analyse the situation according to the three different possible alternatives.
\begin{enumerate}
\item If $w(\sigma) \leq 0$ it is immediate to see
\begin{displaymath}
\int_{v_j^\star(\sigma)}^{w(\sigma)}h(\tau)\, d\tau=\int_{0}^{w(\sigma)}f(0)\,d\tau\leq 0.
\end{displaymath}
\item If $t_j<w(\sigma)\leq t_j'$ we can show
\begin{align*}
\int_{v_j^\star(\sigma)}^{w(\sigma)}h(\tau)\,d\tau  \leq 0.
\end{align*}
arguing similarly to Step 3 in Theorem \ref{th1}.
\item If $w(\sigma)>t_j'$ we obtain
\begin{align} \label{eq14}
\int_{v_j^\star(\sigma)}^{w(\sigma)}h(\tau)\,d\tau & =\int_{t_j}^{w(\sigma)}h(\tau)\,d\tau \\ \notag
& \leq \left| \int_{t_j}^{w(\sigma)}h(\tau)\,d\tau \right| \leq K_2 \left[ (w(\sigma)-t_j)+\frac{1}{q+1}(w(\sigma)^{q+1}-t_j^{q+1}) \right]
\end{align}
At this point set
\begin{displaymath}
\tilde{C}:=K_2\Vert \alpha \Vert_{L^{\infty}(\cM)} \sup_{t \geq t_j'}\frac{(q+1)(t-t_j)+(t^{q+1}-t_j^{q+1})}{(t-t_j)^{q+1}}.
\end{displaymath}
From this and \eqref{eq14} we have
\begin{align} \label{eq15}
\int_\cM \alpha(\sigma) \left(\int_{v_j^\star(\sigma)}^{w(\sigma)}h(\tau)\,d\tau \right) \, dv_g & \leq \Vert \alpha \Vert_{L^{\infty}(\cM)} \int_\cM \left(\int_{v_j^\star(\sigma)}^{w(\sigma)}h(\tau)\, d\tau \right) \, dv_g \\\nonumber
& \leq \tilde{C} \int_\cM (w(\sigma)-t_j)^{q+1}\, dv_g.
\end{align}
Denote
\begin{displaymath}
\tilde{\gamma}:= \sup_{w \in H^1_{G}(\cM)\setminus \left\lbrace 0 \right\rbrace} \frac{\Vert w \Vert_{L^{q+1}(\cM)}}{\Vert w \Vert}
\end{displaymath}
and observe that is finite for Lemma \ref{lemma1}. From \eqref{eq15} we deduce
\begin{equation}
\int_\cM \alpha(\sigma) \left(\int_{v_j^\star(\sigma)}^{w(\sigma)}h(\tau)\,d\tau \right) \, dv_g \leq \tilde{C} \tilde{\gamma}^{q+1} \Vert w-v_j^\star\Vert^{q+1}.
\end{equation}
\end{enumerate}
At this point the conclusion is achieved as in Step 3 of Theorem \ref{th1}.

\textbf{Step 4} We have that
\begin{displaymath}
\liminf_{j \to \infty} \iota_j^{G} =-\infty
\end{displaymath}
Replacing $(f_1)$ with $(f_2')$ and repeating the calculations done in Step 5 of Theorem \ref{th1} we can find a constant $\tilde{\kappa}>0$ and a divergent sequence $(\xi_k)_k$ such that
\begin{displaymath}
J_{G,\lambda}(\xi_k \vartheta_{a,b}^{\varepsilon_0}) < -\tilde{\kappa}\Vert \xi_k \vartheta_{a,b}^{\varepsilon_0}\Vert^2
\end{displaymath}
for $k \geq k_0$ (see the proof of Theorem \ref{th1} for the definition of $\vartheta_{a,b}^{\varepsilon}$). At this point, we notice that we can find a subsequence $(t_{j_k'})_k$ so that $t_{j_k'} \geq \xi_k$ and $\xi_k \vartheta_{a,b}^{\varepsilon_0} \in \mathbb{E}_{j_k}^{G}$. Then
\begin{displaymath}
\lim_{k \to \infty} \iota_{j_k}^{G} \leq  \lim_{k \to \infty }J_{G,\lambda}(\xi_k \vartheta_{a,b}^{\varepsilon_0}) <-\lim_{k \to \infty} \tilde{\kappa}\Vert \xi_k \vartheta_{a,b}^{\varepsilon_0}\Vert^2=-\infty.
\end{displaymath}
From this, we can conclude using the definition of inferior limit getting
\begin{displaymath}
\liminf_{j \to \infty} \iota_{j}^{G}=-\infty
\end{displaymath}
In order to conclude the proof it sufficient to argue as in Step 6 of Theorem \ref{th1} proving that $J_{G,\lambda}$ is invariant under the action of the group $G$ and applying the Principle of Symmetric Criticality of Palais.
\end{proof}
To conclude we exhibit an example of a nonlinearity that satisfies hypothesis $(f_0')$-$(f_2')$.

\begin{example}
 Consider the function
\begin{displaymath}
f(t):=
\begin{cases}
\displaystyle\frac{2 (d-1)  }{d-2} t^{{\frac{ d}{d-2}}}\sin \left(\sqrt[3]{t}\right)+\frac{1}{3}
   t^{\frac{2(2d-1)}{3(d-2)}} \cos \left(\sqrt[3]{t}\right) & \mbox{if $t \geq 0$}  \\
0 & \mbox{if $t<0$}
\end{cases}
\end{displaymath}
whose primitive is
\begin{displaymath}
F(t)=
\begin{cases}
t^{2\frac{d-1}{d-2}} \sin \left(\sqrt[3]{t} \right) & t \geq 0  \\
0 & t<0.
\end{cases}
\end{displaymath}
Hypothesis $(f_0')$ is trivially satisfied since the trigonometric functions are bounded and
\[
\frac{d}{d-2}<2^*-1 \quad \quad \mbox{and} \quad \quad \frac{2(2d-1)}{3(d-2)}<2^*-1.
\]
In order to see the validity of $(f_1')$ one can choose for instance
\[
t_j:= \left[\frac{\pi}{2}(1+4j)\right]^3 \quad \quad \mbox{and} \quad \quad t_j':=\left[\frac{\pi}{2}(3+4j)\right]^3.
\]
It is easy to check that $F$ is decreasing in the interval $[t_j,t'_j]$, hence
\[
	F(t_j) = \sup_{t \in [t_j,t'_j]} F(t).
\]
To prove that $f$ satisfies $(f'_2)$, we choose $\xi_j = t_j \to +\infty$, so that
\[
	\lim_{j \to +\infty} \frac{F(\xi_j)}{\xi_j^2} = \lim_{j \to +\infty} \frac{\xi_j^{2\frac{d-1}{d-2}}}{\xi_j^2} = \lim_{j \to +\infty} \xi_j ^{\frac{2}{d-2}} = +\infty.
\]
Moreover,
\[
	\inf_{t \in [0,\xi_j]} F(t) = F \left( t_{j-1}' \right) = - \left( t'_{j-1} \right)^{2 \frac{d-1}{d-2}} \geq - \left( \xi_{j} \right)^{2 \frac{d-1}{d-2}} = - F ( \xi_j ),
\]
which shows that $(f'_2)$ is verified with $K_3 = 1$.
\end{example}

\bibliographystyle{amsplain}
\bibliography{Oscillatingnonlinearities}
\end{document}